\documentclass[10pt]{amsart}

\usepackage{amssymb,latexsym, mathtools,tikz}
\usepackage{enumerate}
\usepackage{graphicx}
\usepackage{float}
\usepackage{placeins}
\usepackage{mdframed}
\usepackage{amssymb}
\usepackage{esint}
\usepackage{cool}
\usepackage[all,cmtip]{xy}
\usepackage{mathtools}
\usepackage{amstext} 
\usepackage{array}   
\usepackage[shortlabels]{enumitem}
\usepackage{ytableau}
\usepackage{tikz}
\usetikzlibrary{arrows,decorations.markings, cd}
\tikzstyle arrowstyle=[scale=1]
\tikzstyle directed=[postaction={decorate,
decoration={markings,mark=at position .65 with {\arrow[arrowstyle]{stealth}}}}]
\usepackage{mathdots}
\usepackage{quiver}

\newcolumntype{L}{>{$}l<{$}} 
\newcolumntype{C}{>{$}c<{$}}
\newtheorem{theorem}{Theorem}[section]
\newtheorem{lemma}[theorem]{Lemma}
\newtheorem{cor}[theorem]{Corollary}
\newtheorem{prop}[theorem]{Proposition}
\newtheorem{setup}[theorem]{Setup}
\theoremstyle{definition}
\newtheorem{definition}[theorem]{Definition}
\newtheorem{example}[theorem]{Example}

\newtheorem{obs}[theorem]{Observation}
\newtheorem{notation}[theorem]{Notation}

\theoremstyle{remark}
\newtheorem{remark}[theorem]{Remark}

\newtheorem{the context}[theorem]{The Context}

\numberwithin{equation}{theorem}
\numberwithin{equation}{section}

\newcommand{\cat}[1]{\mathcal{#1}}

\newcommand{\tor}{\operatorname{Tor}}

\newcommand{\Ker}{\operatorname{Ker}}

\newcommand{\ideal}[1]{\mathfrak{#1}}
\newcommand{\m}{\ideal{m}}

\newcommand{\bbp}{\mathbb{P}}

\renewcommand{\geq}{\geqslant}
\renewcommand{\leq}{\leqslant}
\renewcommand{\ker}{\Ker}

\newcommand{\maps}[5]{\xymatrix{#1 \ar[r]^-{#3} & #2 \\
#4 \ar@{|->}[r] & #5 \\}}

\newcommand{\mfa}{\mathfrak{a}}

\def\w{\wedge}

\renewcommand{\aa}{\textbf{a}}

\newcommand{\inn}{\textrm{in}}

\begin{document}
\title{Detecting Golodness via Gr\"obner Degeneration}

\author{Keller VandeBogert}
\date{\today}

\maketitle

\begin{abstract}
    In this paper we study the extent to which Golodness may be transferred along morphisms of DG-algebras. In particular, we show that if $I$ is a so-called fiber invariant ideal, then Golodness of $I$ is equivalent to Golodness of the initial ideal of $I$. We use this to transfer Golodness results for monomial ideals to more general classes of ideals. We also prove that any so-called rainbow monomial ideal with linear resolution defines a Golod ring; this result encompasses and generalizes many known Golodness results for classes of monomial ideals. We then combine the techniques developed to give a concise proof that maximal minors of (sparse) generic matrices define Golod rings, independent of characteristic. 
\end{abstract}

\section{Introduction}

Massey products are higher order (co)homological operations that take as inputs a tuple of homology classes and output a subset of $H_\bullet (A)$, where $A$ is some DG-algebra. In general, Massey products can be very difficult to compute explicitly and may only be partially defined as a function. Despite their complexity, these operations have warranted a close study due to their description in terms of differentials arising from associated Eilenberg-Moore spectral sequences (see \cite{gugenheim1974theory}). If all higher order Massey products are trivial (that is, they are defined and only contain $0$), it is known that the associated Eilenberg-Moore spectral sequences degenerate on the first page; this connection was used by Gugenheim and May to carry out explicit cohomological computations of spaces for which more elementary methods did not apply.

In \cite{golod1962homologies}, Golod proved that triviality of Massey products on the Tor-algebra of a local ring $(R,\m,k)$ was intimately connected to the Poincar\'e series $P^R_k (t)$. More precisely, Golod noticed that the existence of so-called \emph{trivial Massey operations} is equivalent to $P^R_k (t)$ attaining a coefficient-wise inequality, originally established by Serre. Such rings are now called \emph{Golod rings}, and despite their long history remain relatively misunderstood even for seemingly simple classes of rings.

Even for rings defined by monomial ideals, it is a difficult problem (with a rather bumpy history) to prove Golodness. Counterexamples to results of Berglund and J\"ollenbeck \cite{berglund2007golod} due to De Stefani \cite{de2016products} and Katth\"an \cite{katthan2017non} indicated that there were gaps in the proofs of some of the results of \cite{berglund2007golod}, and has led to more recent work exploring the extent to which the results of Berglund and J\"ollenbeck remain valid (see \cite{katthan2017non}, \cite{dao2020monomial}, and \cite{vandebogert2021products}). It is worth noting that understanding the Golodness properties of rings defined by monomial ideals has connections to understanding cohomological properties of more geometric objects such as toric varieties, since it is known that the integral cohomology algebra of the associated \emph{moment-angle} complex is isomorphic to the Tor-algebra of certain associated Stanley-Reisner ideals (for more on this perspective, see \cite{denham2007moment}).  

The goal of this paper is to study the extent to which Golodness is preserved upon deforming a ring. One of the main motivations for this is the observation that there are many results related to Golodness that are known only for monomial ideals, implying that there may be cases for which Golodness can be transferred along an associated Gr\"obner degeneration. We also prove that a large class of monomial ideals known as \emph{rainbow monomial ideals} with linear resolution are always Golod. This produces a large class of Golod rings defined by monomial ideals that encompasses and generalizes many known classes of rings throughout the literature. We combine this result with results established by Conca and Varbaro on squarefree Gr\"obner degenerations to prove that ideals of maximal minors of sparse generic matrices define Golod rings; this generalizes a result of Srinivasan using techniques that require no characteristic assumptions on the base field. 

The paper is organized as follows. In Section \ref{sec:backgroundAndTrivialMassey}, we introduce some necessary background and establish notation for the rest of the paper. We also prove some results related to the transfer of trivial Massey operations along certain types of morphisms of DG-algebras which will be employed in later sections (see Theorem \ref{thm:liftingMasseyToDeform} and Proposition \ref{prop:specializeMassey}). In Section \ref{sec:CIvarDiff} we study the case of transferring trivial Massey operations along morphisms $\phi : R \to S$ of polynomial rings whose kernel is generated by variable differences; in particular, the trivial Massey operation is directly induced by the map $\phi$ in such cases.

In Section \ref{sec:fiberInvGrobner}, we examine the relationship between Golodness of an ideal and its initial ideal. We show that ideals with \emph{fiber invariant} Betti numbers (see Definition \ref{def:fiberInv}) have the property that they are Golod if and only if the initial ideal is Golod (see Theorem \ref{thm:golodForFiberInv}). We also use results of Conca and Varbaro to give simple classes of ideals with fiber invariant Betti numbers (see Theorem \ref{thm:fiberInvLinear}). 

In Section \ref{sec:golodRainbows}, we consider a ubiquitous class of monomial ideals which have been referred to as \emph{rainbow monomial ideals} or facet ideals associated to $n$-partite $n$-uniform clutters. These ideals have already been studied for their connections to polarizations of Artinian monomial ideals (\cite{almousa2022polarizations}, \cite{almousa2021polarizations}), arithmetically Cohen-Macaulay sets of points in products of projective space, and the combinatorial structure of the homological invariants associated to such ideals (\cite{nematbakhsh2021linear}, \cite{vandebogert2021linear}). We prove that any rainbow monomial ideal with linear resolution is Golod; this result encompasses many of the Golodness results known for monomial ideals in the literature, since numerous classes of monomial ideals are obtained by specializing rainbow monomial ideals with linear resolution (see the statement of Corollary \ref{cor:golodMonomialIdeals} for a list of these classes of monomial ideals). 

Finally, in Section \ref{sec:golodnessSparseMatrices} we combine many of the results established earlier in the paper to give a concise proof that any ideal generated by the maximal minors of a (sparse) generic matrix is Golod. This extends a result due to Srinivasan, who proved that ideals of maximal minors are Golod in characteristic $0$. One of the ingredients of the proof is also of independent interest: more precisely, we show that any power of an ideal of maximal minors of a (sparse) generic matrix has fiber invariant Betti numbers, generalizing a result of Boocher \cite{boocher2012free}.

\section{Background and the Transfer of Trivial Massey Operations}\label{sec:backgroundAndTrivialMassey}

In this section, we first recall some definitions and establish notation essential to the remainder of the paper. This includes the definition of a DG-algebra and Golodness in terms of trivial Massey operations. We also prove two main results about the transfer of trivial Massey operations along certain classes of morphisms of DG-algebras which we will use for later sections to deduce Golodness. 

\begin{notation}
The notation $(F_\bullet , d_\bullet)$ or $(F_\bullet , d^F_\bullet)$ will denote a complex $F_\bullet$ with differentials $d_\bullet$ or $d^F_\bullet$, respectively (the notation $d^F_\bullet$ is used to distinguish differentials when multiple different complexes are present). When no confusion may occur, the more concise notation $F$ may be written, where the notation $d^F$ is understood to mean the differential of $F$ (in the appropriate homological degree); that is, $d^F (f_i) = d^F_i (f_i)$, where $f_i \in F_i$.

Given a complex $F_\bullet$ as above, elements of $F_n$ will often be written with a subscript $n$, for example as $f_n$, without specifying that they are in $F_n$.
\end{notation}

\begin{definition}\label{def:dga}
A \emph{differential graded algebra} $(F,d)$ (or \emph{DG-algebra}) over a commutative Noetherian ring $R$ is a complex of finitely generated free $R$-modules with differential $d$ and with a unitary, associative multiplication $F \otimes_R F \to F$ satisfying
\begin{enumerate}[(a)]
    \item $F_i F_j \subseteq F_{i+j}$,
    \item $d_{i+j} (f_i f_j) = d_i (f_i) f_j + (-1)^i f_i d_j (f_j)$,
    \item $f_i f_j = (-1)^{ij} f_j f_i$, and
    \item $f_i^2 = 0$ if $i$ is odd,
\end{enumerate}
where $f_k \in F_k$.
\end{definition}

In this paper, all DG-algebras will be assumed to be associative. The next definition will be essential for defining Golod rings. If $z_\lambda$ is a cycle in some complex $A$, then the notation $[z_\lambda]$ denotes the homology class of $z_\lambda$. In the following definition, $\overline{a} := (-1)^{|a|+1} a$, where $|a|$ denotes the homological degree of $a \in A$.  

\begin{definition}
Let $A$ be a DG-algebra with $H_0 (A) \cong k$. Then $A$ admits a \emph{trivial Massey operation} if for some $k$-basis $\mathcal{B} = \{ h_\lambda \}_{\lambda \in \Lambda}$, there exists a function
$$\mu : \coprod_{i=1}^\infty \cat{B}^i \to A$$
such that
\begingroup\allowdisplaybreaks
\begin{align*}
    &\mu ( h_\lambda) = z_\lambda \quad \textrm{with} \quad [z_\lambda] = h_\lambda, \ \textrm{and} \\
    &d \mu (h_{\lambda_1} , \dots , h_{\lambda_p} ) = \sum_{j=1}^{p-1} \overline{\mu (h_{\lambda_1} , \dots , h_{\lambda_j})} \mu (h_{\lambda_{j+1}} , \dots , h_{\lambda_p}). 
\end{align*}
\endgroup
\end{definition}

Observe that taking $p=2$ in the above definition yields that $H_{\geq 1} (A)^2 = 0$, so the induced algebra structure on $H(A)$ is trivial for a DG-algebra admitting a trivial Massey operation.

\begin{notation}
Throughout the paper, if $R$ is a regular local ring or standard graded $k$-algebra, the notation $K^R$ will denote the Koszul complex on the minimal generating set of the (graded) maximal ideal of $R$.
\end{notation}

We can now state the definition of a Golod ring.

\begin{definition}\label{def:Golod}
Let $(R,\m)$ be a local ring and let $K^R$ denote the Koszul complex on the generators of $\m$. If $K^R$ admits a trivial Massey operation $\mu$, then $R$ is called a \emph{Golod ring}. An ideal $I$ in some ring $Q$ will be called \emph{Golod} if the quotient $Q/I$ is Golod.
\end{definition}

\begin{remark}
Note that even though trivial Massey operations are defined with respect to some choice of basis, it is known that \emph{every} choice of basis for the Koszul homology of a Golod ring admits a trivial Massey operation \cite[Remark 5.2.3]{avramov1998infinite}. 

Equivalently, Golodness may be defined by saying that the multiplication is trivial and all higher order Massey products are defined and only contain the element $0$; we will use this perspective more explicitly in Section \ref{sec:golodRainbows}, see Remark \ref{rk:GolodAndMassey} in particular.
\end{remark}

The following result of Herzog and Maleki gives a large class of ideals that are known to be Golod; it is worth noting that the same result holds in characteristic $0$ for arbitrary ideals by a result of Herzog and Huneke (see \cite{herzog2013ordinary}). This result will be needed later in Section \ref{sec:golodnessSparseMatrices}.

\begin{theorem}[{\cite[Theorem 3.1(d)]{herzog2018koszul}}]\label{thm:herzogMaleki}
Let $I \subset R$ be a proper monomial ideal in a polynomial ring $R$. Then $I^n$ and $I^{(n)}$ are Golod for all $n \geq 2$. 
\end{theorem}

For the remainder of this section, we consider the following question: suppose that $\phi : A \to B$ is a morphism of DG-algebras and either $A$ or $B$ admits a trivial Massey operation with respect to some basis. What conditions are needed on $\phi$ in order to transfer this Massey operation to $B$ or $A$, respectively? 

This question is relatively straightforward if $A$ has a trivial Massey operation, but is a little more subtle when one wants to ``pull back" a trivial Massey operation from $B$. We first introduce the following notation.

\begin{notation}
Given a DG-algebra $A$ with $H_0 (A) \cong k$, where $k$ is some field, the notation $\cat{B}^A$ will denote any subset $\cat{B}^A \subset H_\bullet (A)$ forming a $k$-basis for the homology algebra $H_\bullet (A)$. 

If $\phi : A \to B$ is a morphism of DG-algebras, the notation $\phi (\cat{B}^A)$ will denote the set of all elements of the form $\phi_* (h)$, where $h \in \cat{B}^A$ and $\phi_* : H_\bullet (A) \to H_\bullet (B)$ is the induced map on homology.
\end{notation}

The following theorem gives conditions for which a trivial Massey operations can be pulled back along a morphism of DG-algebras.

\begin{theorem}\label{thm:liftingMasseyToDeform}
Let $\phi : A \to B$ be a morphism of DG-algebras inducing a surjection on cycles and an injection on homology. Assume that there exists a trivial Massey operation $\mu$ on some basis $\cat{B}^B$ of $B$ with $\phi (\cat{B}^A) \subset \cat{B}^B$. Then there exists a trivial Massey operation $\mu'$ on $A$ such that
$$\phi \circ \mu' (h_{\lambda_1} , \dots , h_{\lambda_p} ) = \mu (\phi (h_{\lambda_1}) , \dots , \phi (h_{\lambda_p}) ),$$
for all choices of $\lambda_1 , \dots , \lambda_p$.
\end{theorem}

\begin{proof}
Proceed by induction on $p$. Let $p=1$ and $\{ h_\lambda \}_{\lambda \in \Lambda}$ denote the basis $\cat{B}^A$, and define $\mu' (h_{\lambda_1} ) := z_{\lambda_1}$, where $\phi ( z_{\lambda_1} ) = \mu (\phi (h_{\lambda_1}))$ (the element $z_{\lambda_1}$ exists since the induced map $\phi : Z_\bullet(A) \to Z_\bullet (B )$ is assumed to be surjective). By construction, $\phi \circ \mu'( h_{\lambda_1}) = \mu (\phi (h_{\lambda_1}))$. 

Assume now that $p >1$. For any choice of $\lambda_1 , \dots , \lambda_p$, consider the element 
$$m := \sum_{i=1}^{p-1} \overline{\mu' (h_{\lambda_1} , \dots , h_{\lambda_i})} \mu' ( h_{\lambda_{i+1}} , \dots , h_{\lambda_p}).$$
By the inductive hypothesis, $\phi (m) = d \mu (\phi (h_{\lambda_1}) , \dots , \phi (h_{\lambda_p}))$, and since $\phi$ induces an injection on homology, $m = d^A (a)$ for some $a \in A$. By selection of $a$, this implies that 
$$\phi (a) - \mu (\phi (h_{\lambda_1}) , \dots , \phi (h_{\lambda_p} ) )\in Z_\bullet (B).$$
Since $\phi : Z_\bullet (A) \to Z_\bullet (B)$ is surjective, there exists $a' \in Z_\bullet (A)$ such that
$$\phi (a - a') = \mu (\phi (h_{\lambda_1}) , \dots , \phi (h_{\lambda_p} )).$$
Defining $\mu' (h_{\lambda_1} , \dots , h_{\lambda_p} ) := a - a'$, it follows that $\mu'$ satisfies the required properties.
\end{proof}

The following proposition is stated here to give a more specialized class of morphisms of DG-algebras for which the hypotheses of Theorem \ref{thm:golodAndDeformation} are satisfied.

\begin{prop}\label{prop:qisoandSurj}
Let $\phi : A \to B$ be a surjective quasi-isomorphism of complexes. Then the induced map $\phi : Z_\bullet (A) \to Z_\bullet (B)$ is a surjection.
\end{prop}

\begin{proof}
Let $b_i \in Z_i (B)$ be any cycle. By assumption, $\phi (a_i) = b_i$ for some $a_i \in A_i$. It remains to show that $a_i$ is in fact a cycle. The assumption that $\phi$ is a quasi-isomorphism implies that there also exists $a_i' \in Z_i (A)$ and $b_{i+1} \in B_{i+1}$ such that $b_i = \phi (a_i') + d^B (b_{i+1})$, so that $\phi (a_i - a_i') = d^B (b_{i+1})$. Since $\phi_*$ is in particular an injection, this implies that $a_i - a_i' = d^A (a_{i+1})$ for some $a_{i+1} \in A_{i+1}$. This of course implies that $a_i \in Z_i (A)$. 
\end{proof}

The following proposition gives conditions for pushing forward a trivial Massey operation along a morphism of DG-algebras.

\begin{prop}\label{prop:specializeMassey}
Let $\phi : A \to B$ be a quasi-isomorphism of DG-algebras. If $\mu$ is a trivial Massey operation with respect to some basis $\cat{B}^A$ of $A$, then the function $\mu'$ defined by
$$\mu' (\phi(h_{\lambda_1}) , \dots , \phi (h_{\lambda_p}) ) := \phi \big( \mu (h_{\lambda_1} , \dots , h_{\lambda_p}) \big)$$
is a trivial Massey operation on $B$ with respect to the basis $\phi ( \cat{B}^A )$. 
\end{prop}

\begin{proof}
Since $\phi$ is a quasi-isomorphism, $\phi ( \cat{B}^A )$ is a basis for the homology algebra $H_\bullet (B)$. Using the fact that $\phi$ is a morphism of DG-algebras, one finds:
\begingroup\allowdisplaybreaks
\begin{align*}
    d^B \mu' (\phi (h_{\lambda_1}) , \dots , \phi (h_{\lambda_p}) ) &= d^B \phi  ( \mu (h_{\lambda_1} , \dots , h_{\lambda_p} ) \\
    &= \phi d^A \mu  (h_{\lambda_1} , \dots , h_{\lambda_p} ) \\
    &= \phi \Big( \sum_{i=1}^{p-1} \overline{\mu (h_{\lambda_1} , \dots , h_{\lambda_i} )} \mu ( h_{\lambda_{i+1}} , \dots , h_{\lambda_p} ) \Big) \\
    &= \sum_{i=1}^{p-1} \overline{\mu' (\phi (h_{\lambda_1}) , \dots , \phi (h_{\lambda_i}) )} \mu' (\phi (h_{\lambda_{i+1}}) , \dots , \phi (h_{\lambda_p}) ).
\end{align*}
\endgroup
Thus, $\mu'$ is a trivial Massey operation as desired.
\end{proof}

\section{Complete Intersections Generated by Variable Differences}\label{sec:CIvarDiff}

In this section, we consider transferring trivial Massey operations along morphisms of polynomial rings with kernel generated by variable differences. In particular, the results of Section \ref{sec:backgroundAndTrivialMassey} imply that Golodness can be transferred in either direction along such a morphism. This result in particular encompasses polarizations of monomial ideals, which for convenience we define below.

\begin{definition}
    Let $I$ be a monomial ideal in the polynomial ring $S = k[x_1,\ldots, x_n]$, such that for every index $i$, $x_i$ appears with power at most $m_i$ in any minimal generator of $I$. Let $\check{X}_i = \{x_{i1},x_{i2},\ldots, x_{i m_i}\}$ be a set of variables, and let $\tilde S = k[\check{X}_1,\ldots \check{X}_n]$ be the polynomial ring in the union of these variables. An ideal $\tilde I\subset \tilde S$ is a \emph{polarization} of $I$ if
\begin{align*}
\sigma = \{x_{11}-x_{12}, x_{11}-x_{13},\ldots,x_{11}-x_{1m_1}\} \cup \{x_{21}-x_{22},\ldots,x_{21}-x_{2m_2}\}\cup \ldots \\
\cup \{x_{n1}-x_{n2},\ldots, x_{n1}-x_{n,m_n}\}
\end{align*}
is a regular sequence in $\tilde S / \tilde I$ and $\tilde I \otimes \tilde S / \sigma \cong I$.
    \end{definition}

\begin{obs}\label{obs:indSurjKosz}
Let $\phi : R \to S$ be a surjective graded morphism of standard-graded $k$-algebras. Then $\phi$ induces a surjective morphism of DG-algebras $K^R \to K^S$. 
\end{obs}

\begin{proof}
Any graded $k$-algebra morphism is equivalently determined by the data of a morphism of vector spaces $R_1 \to S_1$, and this map can be extended to a morphism of algebras on the corresponding exterior algebras.
\end{proof}

\begin{setup}\label{setup:deformationSetup}
Let $R$ and $S$ be standard-graded polynomial rings over a fixed field $k$, and let $\phi : R \to S$ denote any graded surjective morphism of $k$-algebras with the property that $\mfa := \ker \phi$ is a complete intersection.  
\end{setup}

\begin{remark}
Notice that the assumption that $\phi : R \to S$ is a graded morphism of $k$-algebras in Setup \ref{setup:deformationSetup} implies that $\mfa$ is generated by linear forms.
\end{remark}

\begin{theorem}\label{thm:golodAndDeformation}
Adopt notation and hypotheses as in Setup \ref{setup:deformationSetup}. Let $I \subset R$ be any ideal such that $\mfa$ is regular on $R/I$. Then $I$ is Golod if and only if $\phi (I)$ is Golod.
\end{theorem}

\begin{remark}
Theorem \ref{thm:golodAndDeformation} can also be deduced in the local case by considering the Poincar\'e series of the respective rings; see for instance \cite[Proposition 5.2.4]{avramov1998infinite}. One advantage of the argument given here is that the results of Section \ref{sec:backgroundAndTrivialMassey} tell one precisely how the Massey product is induced by the map $\phi$, which is useful for explicit computations.
\end{remark}

\begin{proof}
$\implies:$ Assume that $I$ is Golod. Since $\mfa$ is regular on $R/I$, the induced map $R/I \otimes_R K^R \to S/\phi(I) \otimes_S K^S$ is a quasi-isomorphism of DG-algebras, so $\phi(I)$ is Golod by Proposition \ref{prop:specializeMassey}.

$\impliedby:$ Assume that $\phi (I)$ is Golod. Since $\mfa$ is regular on $R/I$, the induced map $\widehat{\phi} : R/I \otimes_R K^R \to S/\phi(I) \otimes_S K^S$ is a quasi-isomorphism of DG-algebras, and Observation \ref{obs:indSurjKosz} implies that this induced map is also a surjection. Let $\cat{B}^{R/I}$ denote any $k$-basis for the homology algebra. Then $\phi ( \cat{B}^{R/I} )$ may be extended to a $k$-basis $\cat{B}^{S / \phi (I)}$ for $H_\bullet (S / \phi (I) \otimes_S K^S )$, and the Golodness assumption of $\phi (I)$ implies that there exists a trivial Massey operation with respect to $\cat{B}^{S / \phi (I)}$. Combining Proposition \ref{prop:qisoandSurj} with Theorem \ref{thm:liftingMasseyToDeform}, it follows that $I$ is Golod. 
\end{proof}

\begin{cor}
Let $\widetilde{I}$ denote the polarization of any monomial ideal $I$. Then $I$ is Golod if and only if $\widetilde{I}$ is Golod.
\end{cor}

\section{Fiber Invariant Gr\"obner Degenerations}\label{sec:fiberInvGrobner}

In this section we study one of the main motivating questions of this paper: to what extent can Golodness be transferred along an associated Gr\"obner degeneration? We first give the necessary background on Gr\"obner degenerations, and use this material to define the notion of \emph{fiber invariant} Betti numbers (with respect to a given term order). We then use results of Conca and Varbaro from \cite{conca2020square} to furnish some simple classes of ideals with fiber invariant Betti numbers. Finally, we prove that for an ideal $I$ with fiber invariant Betti numbers, Golodness of $I$ is equivalent to Golodness of the initial ideal $\inn_< (I)$. This result will be particularly useful in Section \ref{sec:golodnessSparseMatrices}.

The results on Gr\"obner degenerations are stated here without proof; for a more complete introduction to this material see \cite{green1998generic} and the references therein.

\begin{setup}\label{set:homogenizationSetup}
Let $R = k[x_1 , \dots , x_n]$ and fix a weight vector $w \in\mathbb{N}^{n}$ inducing a term order $<$. Define $R^h := R[t]$. Let $f \in R$ and write
\[f=\underset{1\leq i\leq l}{\sum}c_ix^{\mathbf{a}_i}\]
with $w( \aa_1) \geq w(\aa_2) \geq \cdots \geq w(\aa_l)$. Define the homogenization of $f$ with respect to $w$ as
$$f^h := t^{w(\aa_1)} \cdot f(t^{-w_1} x_1 , \dots , t^{-w_n} x_n).$$
For any homogenenous ideal $I \subset R$, define the homogenization of $I$ via
$$I^h := \{ f^h \mid f \in I \}.$$
\end{setup}

\begin{obs}\label{obs:2gradings}
Adopt notation and hypotheses as in Setup \ref{set:homogenizationSetup}. If $I \subset R$ is a homogeneous ideal, then $I^h$ is homogeneous with respect to the following gradings:
\begin{enumerate}
    \item $\deg (x_i) = 1$ for $1 \leq i \leq n$ and $\deg (t) = 0$, or
    \item $\deg (x_i) = w_i$ for $1 \leq i \leq n$ and $\deg (t) = 1$. 
\end{enumerate}
\end{obs}

By convention, we will always think of a homogenized ideal $I^h$ as being graded with respect to the grading $(2)$ of Observation \ref{obs:2gradings}. The following proposition is a fundamental result in the study of initial ideals; it says that $I$ and $\inn_< I$ can be connected via a flat family.

\begin{prop}
Adopt notation and hypotheses as in Setup \ref{set:homogenizationSetup}. If $G$ is a $w$-Gr\"{o}bner basis of $I$, then:
\begin{enumerate}[(a)]
    \item $I^h = (g^h \mid g \in G)$, and
    \item $R^h / I^h$ is flat over $k[t]$. 
\end{enumerate}
In particular,
$$R^h / I^h \otimes R^h/(t) = R / \inn_w (I) \quad \textrm{and} \quad R^h / I^h \otimes R^h/(t-1) = R/I.$$
\end{prop}

The flatness of $R^h / I^h$ implies that free resolutions over $R^h$ will descend to free resolutions over $R$. This fact is made precise in the following:

\begin{cor}\label{cor:flatResolution}
Adopt notation and hypotheses as in Setup \ref{set:homogenizationSetup}. Let
\begin{align*}
    F: \quad \cdots \to F_i^h \to \cdots \to F^h_1 \to R^h \to 0.
\end{align*}
be a minimal free resolution of $R^h / I^h$. Then this resolution descends to:
\begin{enumerate}
    \item A minimal $R$-free resolution of $R/in_w(I)$, by applying the functor $- \otimes_{k[t]} k[t]/(t)$.
    \item A (possibly non-minimal) $R$-free resolution of $R/I$, by applying the functor $- \otimes_{k[t]} k[t]/(t-1)$
\end{enumerate}
In particular, one has
$$\beta_{ij}^{R^h} (I^h) = \beta_{ij}^R (\inn_w (I)) \quad \textrm{for all} \ i,j.$$
\end{cor}

The next proposition is an evident observation resulting from Corollary 
\ref{cor:flatResolution}; we state it explicitly here because it will be an essential ingredient for later proofs.

\begin{prop}\label{prop:bnosandminimality}
Let $G$ denote a reduced Gr\"obner basis of an ideal $I$, and let $F^h$ denote the homogeneous minimal free resolution of $I^h$. Then $\beta_{ij} (\inn_< I) = \beta_{ij} (I)$ if and only if $F^h \otimes k[t]/(t-1)$ is a minimal complex.
\end{prop}

The following definition introduces a key notion that will allow us to transfer Golodness along associated Gr\"obner degenerations.

\begin{definition}\label{def:fiberInv}
Let $R$ be a standard graded polynomial ring over a field $k$ endowed with some term order $<$. A homogeneous ideal $I$ is said to have \emph{fiber invariant} Betti numbers (with respect to $<$) if
$$\beta_{ij} (R/I) = \beta_{ij} ( R / \inn_< I) \quad \textrm{for all} \ i,j.$$
\end{definition}

It is worth noting that the property $\beta_{ij} (R/I) = \beta_{ij} ( R / \inn_< I)$ for all $i, \ j$ at least implies that $I$ has a minimal generating set that forms a Gr\"obner basis, which is already an uncommon property. That being said, the following theorem will give us a sufficiently abundant source of examples of ideals with fiber invariant Betti numbers; in particular, such ideals exist (the fact that such ideals exist has already been known in the literature, see for example \cite{boocher2012free} and \cite{mohammadi2014divisors}). 

\begin{theorem}\label{thm:fiberInvLinear}
Assume $I$ is a homogeneous ideal with linear minimal free resolution. The ideal $I$ has fiber invariant Betti numbers in each of the following cases:
\begin{enumerate}
    \item the initial ideal $\inn_< (I)$ has linear resolution, or
    \item the initial ideal $\inn_< (I)$ is squarefree.
\end{enumerate}
\end{theorem}

\begin{proof}
\textbf{Proof of $(1):$} If $\inn_< (I)$ has linear resolution, then there are no nontrivial consecutive cancellations (see \cite{peeva2004consecutive}) in the Betti table of $\inn_< (I)$, so the Betti tables must coincide.

\textbf{Proof of $(2)$:} If $\inn_< (I)$ is squarefree, then $I$ also has linear resolution by \cite[Theorem 1.2]{conca2020square}. The proof then follows by $(1)$. 
\end{proof}

The following theorem shows that Golodness may be transferred along Gr\"obner degenerations associated to ideals with fiber invariant Betti numbers.

\begin{theorem}\label{thm:golodForFiberInv}
Let $I$ be any ideal with fiber invariant Betti numbers with respect to some term order $<$. Then,
$$R/I \ \textrm{is Golod} \quad \iff \quad R/\inn_< (I) \ \textrm{is Golod}.$$
\end{theorem}

\begin{proof}
Let $I^h \subset R^h$ denote the associated homogenization with respect to the term order $<$, and let $I^a$ denote $I^h \otimes_{k[t]} k[t]/(t-a)$. The assumption of fiber invariance is equivalent to the statement that there is an isomorphism of $R$ (or $k$)-modules $\tor_\bullet^R (R^h/I^h , k) \cong F^h \otimes_R k \cong  \tor_\bullet^R (R/I^a , k)$ for all $a \in k$, and this isomorphism is furthermore an isomorphism of algebras. This implies that the induced map $R^h / I^h \otimes_R K \to R / I^a \otimes_R K$ is a surjective quasi-isomorphism of DG-algebras, whence $I^h$ is Golod as an $R$-module if and only if $I^a$ is Golod for some $a \in k$ by Proposition \ref{prop:qisoandSurj} and Theorem \ref{thm:liftingMasseyToDeform}. Employing Proposition \ref{prop:specializeMassey}, this is in fact equivalent to $I^a$ being Golod for \emph{all} $a \in k$.
\end{proof}

The standard inequality $\beta_{ij} (R/I) \leq \beta_{ij} (R / \inn_< (I))$ may tempt one to believe that trivial homology algebra (or even Golodness) of $R/ \inn_< (I)$ implies trivial homology algebra/Golodness for $I$, but the following simple example shows that this is not true. 

\begin{example}
Let $R = k[x_1,x_2,x_3]$ and consider the ideal $$I = \left({x}_{1}^{2},{x}_{1}{x}_{3},-{x}_{1}{x}_{2}+{x}_{
        3}^{2},{x}_{2}{x}_{3},{x}_{2}^{2}\right).$$ 
        The ideal $I$ is a grade $3$ Gorenstein ideal and hence is not Golod. With respect to the standard lexicographic order on the variable order $x_1 > x_2 > x_3$, $I$ has initial ideal
        $$\inn_< (I) = \left({x}_{1}^{2},{x}_{1}{x}_{2},{x}_{2}^{2},{
        x}_{1}{x}_{3},{x}_{2}{x}_{3},{x}_{3}^{3}\right),$$
        and one can check using Macaulay2 that $\inn_< (I)$ \emph{is} Golod. In particular, having a Golod initial ideal does not imply that the original ideal is Golod. 
\end{example}

\section{Golodness of Rainbow Monomial Ideals with Linear Resolution}\label{sec:golodRainbows}

In this section, we prove that so-called \emph{rainbow} monomial ideals with linear resolution define Golod rings. This result generalizes and extends many pre-existing Golodness results on monomial ideals in the literature (see Corollary \ref{cor:golodMonomialIdeals}). In order to prove this result, we will need to recall a fair amount of background on the structure of the linear strands of rainbow monomial ideals along with higher Massey products on Koszul homology. 

\begin{definition}
A \emph{rainbow monomial ideal} on $n$ colors is any monomial ideal generated by monomials of the form
$$x_{1 a_1} x_{2 a_2} \cdots x_{na_n}$$
for some choices of $a_1 , \dots , a_n$. Variables of the form $x_{i a_i}$ are said to come from the $i$th \emph{color class}. If there are $d_i$ possible values for $a_i$ in the $i$th color class, then the $i$th color class will be said to have $d_i$ variables. 
\end{definition}

\begin{remark}
A rainbow monomial ideal on $n$ colors can be equivalently formulated as a facet ideal of an $n$-partite $n$-uniform clutter. This is the perspective taken in \cite{nematbakhsh2021linear}. 
\end{remark}

Our main interest in this section will be rainbow monomial ideals having linear resolution. We first adopt the following setup:

\begin{setup}\label{setup:rainSetup}
Let $I$ be a rainbow monomial ideal on $n$ colors with $d_i$ variables in each color class. Given such an ideal and any subset $L \subset [n]$, let $p_L$ denote the ranked projection onto the colors contained in $L$; that is, the ideal obtained by setting all variables of colors classes in $[n] \backslash L$ equal to $1$.

If $A = \{ \aa^1 , \dots , \aa^n \}$, where $\aa^i = \{ a_1^i < \cdots < a_{j_i}^i \}$, then the following notation will be used:
$$x_{[n] , A} := x_{1 \aa^1} \cdot x_{2 \aa^2} \cdots x_{n \aa^n}.$$
Similarly, for basis elements of the Koszul complex the following notation will be used:
$$e_{[n],A} := e_{1 \aa^1} \w e_{2 \aa^2} \w \cdots \w e_{n \aa^n}.$$
A basis element in the Koszul complex will be said to have \emph{valid multidegree} if for all choices of $b_i \in \aa^i$, one has that $x_{1 b_1} x_{2 b_2} \cdots x_{n b_n} \in G(I)$. Given a basis element $e_{[n] , A}$ with valid multidegree, define
$$\eta_{n-i} (e_{[n],A}) := e_{1 \aa^1} \w e_{2 \aa^2} \w \cdots \w d(e_{n-i+1, \aa^{n-i+1}}) \w \cdots \w  d(e_{n \aa^n}),$$
where $d$ denotes the differential of the associated Koszul complex. Notice that $\eta_{n-i} (e_{[n],A} )$ is a cycle in $R/I \otimes_R K_\bullet$ by definition of being a valid multidegree.
\end{setup}

\begin{theorem}[{\cite[Corollary 5.3]{nematbakhsh2021linear}}]\label{thm:aminLinear}
Let $I$ be a rainbow monomial ideal with linear resolution. Then any ranked projection of $I$ also has linear resolution. Likewise, $I$ has linear resolution if and only if its complementary ideal $I^c$ has linear resolution. 
\end{theorem}

\begin{remark}
The notation $I^c$ for the complementary ideal denotes the ideal with $G(I^c) = G(I_1 \cdots I_n) \backslash G(I)$, where $I_i = (x_{i 1 } , \dots , x_{i d_i} )$. 
\end{remark}

In order to study Golodness of rainbow monomial ideals with linear resolution, we will first need an explicit basis for the associated Koszul homology. 

\begin{obs}\label{obs:rainbowKoszul}
Let $I$ be a rainbow monomial ideal with linear resolution. Then a basis for the Koszul homology $H_\bullet (R/I \otimes_R K^R)$ is represented by all 
$$\eta_{n-1} (e_{[n], A}), \quad \textrm{where} \ e_{[n],A} \ \textrm{has valid multidegree.}$$ 
\end{obs}

\begin{proof}
Suppose that $I$ has $n$ colors and $d_i$ variables in each color class. Let $I_i = (x_{i 1 } , \dots , x_{i d_i} )$ for each $1 \leq i \leq n$. Then the minimal free resolution of $I$ is a the subcomplex of the multigraded minimal free resolution of $I_{1} I_2 \cdots I_n$ obtained by restricting to all valid multidegrees. Using the explicit form for the Koszul homology given in, for instance, \cite{vandebogert2020vanishing} the result follows.
\end{proof}

\begin{notation}
Adopt notation and hypotheses as in Setup \ref{setup:rainSetup}. The notation $h_{[n],A}$ will be used to denote the homology class of the element $\eta_{n-1} (e_{[n],A})$. Use the notation $\cat{B}^{R/I}$ to denote the set of all such basis elements.
\end{notation}

The following lemma shows that there exists a partially defined trivial Massey operation with respect to certain tuples of basis elements as in Observation \ref{obs:rainbowKoszul}. We will see that higher Massey products of all other tuples must vanish for multidegree reasons, which will yield the desired Golodness statement.

\begin{lemma}\label{lem:trivialMasseyonValid}
Adopt notation and hypotheses as in Setup \ref{setup:rainSetup}. Let $\cat{V} \subset \coprod_{i=1}^\infty \Big(\cat{B}^{R/I} \Big)^i$ be the set of all tuples $(h_{[n],A_1} , \dots , h_{[n],A_p})$ such that $x_{[n],A_1} \cdots x_{[n],A_p}$ has valid multidegree. Then there exists a function 
$$\mu : \cat{V} \to R/I \otimes_R K^R$$
such that $\mu (h_{[n],A}) = \eta_{n-1} (e_{[n],A})$ and
$$d \mu (h_{[n],A_1} , \dots , h_{[n], A_p} ) = \sum_{j=1}^{p-1} \overline{\mu (h_{[n],A_1} , \dots , h_{[n],A_i}} ) \mu (h_{[n],A_{i+1}} , \dots , h_{[n],A_p}).$$
\end{lemma}

\begin{proof}
Define $\mu$ as follows:
$$ \mu (h_{[n],A_1} , \dots , h_{[n], A_p} ) := \eta_{n-2} (e_{[n],A_1}) \w \eta_{n-1} (e_{[n],A_2}) \w \cdots \w \eta_{n-1} (e_{[n],A_p}).$$
Then a computation identical to that given in \cite[Lemma 4.1]{vandebogert2021products} yields the result.
\end{proof}

Since we will use terminology related to Massey products for the remainder of this section, we give the definition here for convenience.

\begin{definition}\label{def:MasseyProduct}
Let $A$ be a DG-algebra and let $a_1 , \dots , a_n \in H_\bullet (A)$. 

\begin{enumerate}
    \item The \emph{$n$-ary Massey product} $\mu_n (a_1 , \dots , a_n) \subset H_\bullet (A)$ is a partially defined set-valued function. It is defined if there exist $a_{ij} \in A$ for $1 \leq i \leq j \leq n$ such that
    \begin{enumerate}
        \item the homology class of $a_{ii}$ is equal to $a_i$, and
        \item there is an equality 
        $$d a_{ij} = \sum_{k=i}^j \overline{a_{ik}} a_{kj}.$$
    \end{enumerate}
    Then the Massey product is defined to be the set of all homology classes of the sums $\sum_{k=1}^n \overline{a_{1k}} a_{kn}$. 
    \item The DG-algebra $A$ satisfies $(B_r)$ if for all $k \leq r$ and elements $a_1 , \dots , a_k \in H_\bullet (A)$, one has:
\begin{center}
    the Massey product $\mu_k (a_1 , \dots , a_k)$ exists and only contains $0$.
\end{center}
\end{enumerate}
\end{definition}

\begin{remark}
Observe that if $A$ has an internal grading, then any Massey product may be chosen to respect the internal grading; in particular, any Massey product on the Koszul homology of a multigraded algebra may be chosen to respect this multigrading. 
\end{remark}

\begin{remark}\label{rk:GolodAndMassey}
Notice that if a DG-algebra $A$ satisfies $B_r$ for all $r \geq 2$, then by definition for all elements $a_1 , \dots , a_r \in H_\bullet (A)$ there exist $a_{ij} \in A$ such that
$$d a_{1r} = \sum_{i=1}^r \overline{a_{1i}} a_{ir}.$$
Defining $\mu (a_1 , \dots , a_i) := a_{1i}$ for $i \leq r$ and $\mu (a_{i+1} , \dots , a_r) := a_{ir}$ for $i < r$ recovers the definition of a trivial Massey operation, whence Golodness of a ring $R/I$ is equivalent to the algebra $R/I \otimes_R K^R$ satisfying $(B_r)$ for all $r \geq 2$. 
\end{remark}

\begin{lemma}[{\cite[Proposition 2.3]{may1969matric}} or {\cite[Lemma 20]{kraines1966massey}}]\label{lem:inductiveStepForMassey}
Let $A$ be a DG-algebra satisfying $(B_{r-1})$ for some $r \geq 3$. Then all $r$-ary Massey products exist and contain only a single element.
\end{lemma}

The following theorem is the main result of this section.

\begin{theorem}\label{thm:rainbowIsGolod}
Let $I$ be a rainbow monomial ideal with linear resolution. Then $I$ is Golod.
\end{theorem}

\begin{proof}
The result will follow by proving that all higher order Massey products are defined and are trivial; that is, all higher order Massey products contain only the element $0$. To this end, we proceed inductively on the order of the higher Massey product. When $p=2$, the assumption that $I$ has linear resolution implies that the homology algebra $H_\bullet (R/I \otimes_R K^R)$ is trivial.

Assume now that $p >2$. By the inductive hypothesis, the algebra $R/I \otimes_R K^R$ satisfies $(B_{p-1})$, so by Lemma \ref{lem:inductiveStepForMassey} all $p$-ary Massey products are defined and contain only a single element. Let $h_{[n],A_1} , \dots , h_{[n],A_p}$ be a collection of basis elements in $\cat{B}^{R/I}$ with valid multidegree (with respect to the ideal $I$). There are two cases to consider:

\textbf{Case 1:} The element $x_{[n],A_1} \cdots x_{[n],A_p}$ does not have valid multidegree. In this case, since the Massey product $\mu (h_{[n],A_1}) , \dots , h_{[n],A_p} )$ does not have valid multidegree, it represents the trivial homology class.

\textbf{Case 2:} The element $x_{[n],A_1} \cdots x_{[n],A_p}$ does have valid multidegree. In this case, there is a trivial Massey operation on the set of all elements whose product has valid multidegree by Lemma \ref{lem:trivialMasseyonValid}. This implies that the Massey operation must only contain the $0$ element. 
\end{proof}

As previously mentioned, rainbow monomial ideals with linear resolution are ubiquitous in the literature; the following corllary gives a non-exhaustive list of monomial ideals obtained as (specializations of) rainbow monomial ideals with linear resolution (see Section $6$ of \cite{floystad2017letterplace} for more on co-letterplace ideals and their specializations).

\begin{cor}\label{cor:golodMonomialIdeals}
The following classes of monomial ideals are Golod:
\begin{enumerate}
    \item Edge ideals of Ferrers graphs and hypergraphs.
    \item Any restricted power of a strongly stable hypergraph ideal.
    \item Co-letterplace ideals (and by specialization, edge ideals of cointerval $d$-hypergraphs and uniform face ideals).
    \item Clutters associated to arithmetically Cohen-Macaulay sets of points in $(\bbp^1)^{\times d}$.
    \item Alexander duals of polarizations of Artinian monomial ideals.
\end{enumerate}
More generally, any monomial ideal with linear resolution that has a polarization that can be relabelled as a rainbow monomial ideal is Golod. 
\end{cor}

The question of when a given monomial ideal with linear resolution has a polarization that can be relabelled as a rainbow monomial ideal seems to be subtle; the following example shows that there exist monomial ideals with no such polarization. 

\begin{example}[\cite{reiner2001linear}]
Consider the following ideal:
$$I := (x_3 x_4 x_5 x_6, x_2 x_4 x_5 x_6, x_1 x_4 x_5 x_6,$$
 $$x_3 x_4 x_5 x_7, x_2 x_4 x_5 x_v, x_1 x_3 x_5 x_7, x_1 x_2 x_4 x_7, x_1 x_4 x_6 x_7,$$
$$ x_1 x_5 x_6 x_7, x_3 x_4 x_6 x_7, x_2 x_5 x_6 x_7, x_2 x_3 x_6 x_7 , x_1 x_2 x_3 x_7) \subset k[x_1 , \dots , x_7].$$
Then $I$ has linear resolution, but if $\operatorname{char} k  \neq 2, \ 3$, Reiner and Welker show that there is no choice of basis in the minimal free resolution for which the differentials have coefficients $0$ or $\pm 1$. Thus, there does not exist a polarization of $I$ that can be relabelled as a rainbow monomial ideal since the minimal free resolution of a rainbow monomial ideal has coefficients $0$ or $\pm 1$, independent of characteristic.
\end{example}

\section{Golodness of Maximal Minors of Sparse Matrices}\label{sec:golodnessSparseMatrices}

In this section, we combine many of the results established in this paper thus far to prove that (any power of) maximal minors of (sparse) generic matrices define Golod rings. The proof of this result will use some results already established in the literature which we recall for convenience. Using these results we will deduce that these ideals have fiber invariant Betti numbers and hence Golodness may be transferred along to associated Gr\"obner degeneration. The initial ideals in these cases will be evidently Golod by Theorems \ref{thm:rainbowIsGolod} and \ref{thm:herzogMaleki}, whence the result will follow immediately. 

We begin this section by adopting the following setup:

\begin{setup}\label{setup:sparseDet}
Let $X$ be a generic $n \times m$ matrix ($n \leq m$) with coordinate ring $R = k[X] = k[x_{ij}]_{\substack{1 \leq i \leq n \\ 1 \leq j \leq m}}$. Let $X'$ be any matrix obtained by setting entries of $X$ equal to $0$ in such a way that the nonzero entries form a 2-sided ladder (see, for instance, \cite[Remark 4.11]{celikbas2021rees}). The notation $I_n (X')$ will denote the ideal of $n \times n$ (maximal) minors of $X'$.
\end{setup}

The following theorem combines numerous results from across the literature on initial ideals of ideals of maximal minors; we will outline the proofs and references for these implications in Remark \ref{rk:proofForSparse}. 

\begin{theorem}\label{thm:sparseMinorsResults}
Adopt notation and hypotheses as in Setup \ref{setup:sparseDet} and define $I := I_n (X')$. Then:
\begin{enumerate}
    \item For any term order $<$, the set of maximal minors of $X'$ forms a Gr\"obner basis for $I$ (that is, the set of maximal minors is a universal Gr\"obner basis).
    \item For any diagonal term order $<$ and all $t \geq 1$, there is the equality
    $$\inn_< (I^t) = \inn_< (I)^t.$$
    \item For any term order $<$, the ideal $\inn_< (I)^t$ has linear resolution for all $t$. 
\end{enumerate}
\end{theorem}

\begin{remark}\label{rk:proofForSparse}
Part $(1)$ of Theorem \ref{thm:sparseMinorsResults} is a well-known result due to Sturmfels and Zelevinsky, see \cite{sturmfels1993maximal} and \cite{bernstein1993combinatorics}. In the case that $X'$ is a fully generic matrix, part $(2)$ was proved originally by Conca in \cite[Theorem 2.1]{conca1997grobner}. Finally, part $(3)$ is immediate from the results of \cite{celikbas2021rees}, since their results show that the defining ideal of the Rees algebra has a quadratic Gr\"obner basis and is hence Koszul, so the linear powers property follows from \cite[Theorem 1.1]{herzog2004monomial}. 
\end{remark}

\begin{cor}\label{cor:fiberInvSparse}
Adopt notation and hypotheses as in Setup \ref{setup:sparseDet}. Then $I_n (X')$ has fiber invariant Betti numbers for any term order $<$.
\end{cor}

\begin{remark}
In the case that $X'$ is a fully generic matrix, Corollary \ref{cor:fiberInvSparse} was proved by Boocher in {\cite[Theorem 3.1]{boocher2012free}}.
\end{remark}

\begin{proof}
By part $(1)$ of Theorem \ref{thm:sparseMinorsResults}, all initial ideals of $I_n (X')$ are squarefree. Moreover, the pruning technique given by Boocher in \cite{boocher2012free} shows that $I_n (X')$ has linear resolution, so the result follows from Theorem \ref{thm:fiberInvLinear}.
\end{proof}

The following result is the main result of this section and employs the results many of the earlier sections of this paper.

\begin{theorem}\label{thm:sparseGolod}
Adopt notation and hypotheses as in Setup \ref{setup:sparseDet}. Then $I_n (X')^t$ is Golod for all $t \geq 1$. 
\end{theorem}

\begin{remark}
It is worth mentioning that in the case that $X'$ is fully generic and $k$ has characteristic $0$, the fact that $I_n (X')$ is Golod was originally proved by Srinivasan in \cite{srinivasan1989algebra}; this was a corollary of the explicit DG-algebra structure constructed on the Eagon-Northcott complex. This algebra structure was heavily reliant on the characteristic $0$ assumption, however, so it is not clear how similar techniques could be used to prove Theorem \ref{thm:sparseGolod} in arbitrary characteristic. Likewise, in the case that $t >1$, the fact that $I_n(X')^t$ is Golod in characteristic $0$ follows from the results of \cite{herzog2013ordinary}.
\end{remark}

\begin{proof}
Let $I := I_n (X')$. There are two cases to consider: the first case is for $t=1$ and the second is for $t>1$. 

\textbf{Case 1: $t=1$.} In this case, recall that $I$ has fiber invariant Betti numbers with respect to any term order. Thus, Golodness of $I$ is equivalent to Golodness of $\inn_< (I)$. However, $\inn_< (I)$ is a rainbow monomial ideal with linear resolution and hence Golod by Theorem \ref{thm:rainbowIsGolod}.

\textbf{Case 2: $t>1$.} Let $<$ denote the standard diagonal term order. By Theorem \ref{thm:sparseMinorsResults}, the initial ideal of $I^t$ has linear resolution for all $t > 1$. By Theorem \ref{thm:fiberInvLinear} the ideal $I^t$ has fiber invariant Betti numbers, and by Theorem \ref{thm:herzogMaleki} the ideal $\inn_< (I)^t$ is Golod. Employing Theorem \ref{thm:golodForFiberInv}, it follows that $I^t$ is Golod for all $t>1$. 
\end{proof}

\bibliographystyle{amsalpha}
\bibliography{biblio}

\newcommand{\etalchar}[1]{$^{#1}$}
\providecommand{\bysame}{\leavevmode\hbox to3em{\hrulefill}\thinspace}
\providecommand{\MR}{\relax\ifhmode\unskip\space\fi MR }
\providecommand{\MRhref}[2]{%
  \href{http://www.ams.org/mathscinet-getitem?mr=#1}{#2}
}
\providecommand{\href}[2]{#2}
\begin{thebibliography}{CDF{\etalchar{+}}21}

\bibitem[AFL22]{almousa2022polarizations}
Ayah Almousa, Gunnar Fl{\o}ystad, and Henning Lohne, \emph{Polarizations of
  powers of graded maximal ideals}, Journal of Pure and Applied Algebra
  \textbf{226} (2022), no.~5, 106924.

\bibitem[AV21]{almousa2021polarizations}
Ayah Almousa and Keller VandeBogert, \emph{Polarizations and {H}ook
  partitions}, Submitted. arXiv preprint arXiv:2107.07535 (2021).

\bibitem[Avr98]{avramov1998infinite}
Luchezar~L Avramov, \emph{Infinite free resolutions}, Six lectures on
  commutative algebra, Springer, 1998, pp.~1--118.

\bibitem[BJ07]{berglund2007golod}
Alexander Berglund and Michael J{\"o}llenbeck, \emph{On the golod property of
  stanley--reisner rings}, Journal of Algebra \textbf{315} (2007), no.~1,
  249--273.

\bibitem[Boo12]{boocher2012free}
Adam Boocher, \emph{Free resolutions and sparse determinantal ideals},
  Mathematical Research Letters \textbf{19} (2012), no.~4, 805--821.

\bibitem[BZ93]{bernstein1993combinatorics}
David Bernstein and Andrei Zelevinsky, \emph{Combinatorics of maximal minors},
  Journal of Algebraic Combinatorics \textbf{2} (1993), no.~2, 111--121.

\bibitem[CDF{\etalchar{+}}21]{celikbas2021rees}
Ela Celikbas, Emilie Dufresne, Louiza Fouli, Elisa Gorla, Kuei-Nuan Lin,
  Claudia Polini, and Irena Swanson, \emph{Rees algebras of sparse
  determinantal ideals}, arXiv preprint arXiv:2101.03222 (2021).

\bibitem[Con97]{conca1997grobner}
Aldo Conca, \emph{Gr{\"o}bner bases of powers of ideals of maximal minors},
  Journal of Pure and Applied Algebra \textbf{121} (1997), no.~3, 223--231.

\bibitem[CV20]{conca2020square}
Aldo Conca and Matteo Varbaro, \emph{Square-free gr{\"o}bner degenerations},
  Inventiones mathematicae \textbf{221} (2020), no.~3, 713--730.

\bibitem[DDS20]{dao2020monomial}
Hailong Dao and Alessandro De~Stefani, \emph{On monomial golod ideals}, Acta
  Mathematica Vietnamica (2020), 1--9.

\bibitem[DS07]{denham2007moment}
Graham Denham and Alexander~I Suciu, \emph{Moment-angle complexes, monomial
  ideals and massey products}, Pure and Applied Mathematics Quarterly
  \textbf{3} (2007), no.~1, 25--60.

\bibitem[DS16]{de2016products}
Alessandro De~Stefani, \emph{Products of ideals may not be golod}, Journal of
  Pure and Applied Algebra \textbf{220} (2016), no.~6, 2289--2306.

\bibitem[FGH17]{floystad2017letterplace}
Gunnar Fl{\o}ystad, Bj{\o}rn~M{\o}ller Greve, and J{\"u}rgen Herzog,
  \emph{Letterplace and co-letterplace ideals of posets}, Journal of Pure and
  Applied Algebra \textbf{221} (2017), no.~5, 1218--1241.

\bibitem[GM74]{gugenheim1974theory}
Victor~KAM Gugenheim and J~Peter May, \emph{On the theory and applications of
  differential torsion products}, vol. 142, American Mathematical Soc., 1974.

\bibitem[Gol62]{golod1962homologies}
Evgeniy~Solomonovich Golod, \emph{Homologies of some local rings}, Doklady
  Akademii Nauk, vol. 144, Russian Academy of Sciences, 1962, pp.~479--482.

\bibitem[Gre98]{green1998generic}
Mark~L Green, \emph{Generic initial ideals}, Six lectures on commutative
  algebra, Springer, 1998, pp.~119--186.

\bibitem[HH13]{herzog2013ordinary}
J{\"u}rgen Herzog and Craig Huneke, \emph{Ordinary and symbolic powers are
  golod}, Advances in Mathematics \textbf{246} (2013), 89--99.

\bibitem[HHZ04]{herzog2004monomial}
J{\"u}rgen Herzog, Takayuki Hibi, and Xinxian Zheng, \emph{Monomial ideals
  whose powers have a linear resolution}, Mathematica Scandinavica (2004),
  23--32.

\bibitem[HM18]{herzog2018koszul}
J{\"u}rgen Herzog and Rasoul~Ahangari Maleki, \emph{Koszul cycles and golod
  rings}, manuscripta mathematica \textbf{157} (2018), no.~3, 483--495.

\bibitem[Kat17]{katthan2017non}
Lukas Katth{\"a}n, \emph{A non-golod ring with a trivial product on its koszul
  homology}, Journal of Algebra \textbf{479} (2017), 244--262.

\bibitem[Kra66]{kraines1966massey}
David Kraines, \emph{Massey higher products}, Transactions of the American
  Mathematical Society \textbf{124} (1966), no.~3, 431--449.

\bibitem[May69]{may1969matric}
J~Peter May, \emph{Matric massey products}, Journal of Algebra \textbf{12}
  (1969), no.~4, 533--568.

\bibitem[MS14]{mohammadi2014divisors}
Fatemeh Mohammadi and Farbod Shokrieh, \emph{Divisors on graphs, connected
  flags, and syzygies}, International Mathematics Research Notices
  \textbf{2014} (2014), no.~24, 6839--6905.

\bibitem[Nem21]{nematbakhsh2021linear}
Amin Nematbakhsh, \emph{Linear strands of edge ideals of multipartite uniform
  clutters}, Journal of Pure and Applied Algebra \textbf{225} (2021), no.~10,
  106690.

\bibitem[Pee04]{peeva2004consecutive}
Irena Peeva, \emph{Consecutive cancellations in betti numbers}, Proceedings of
  the American Mathematical Society \textbf{132} (2004), no.~12, 3503--3507.

\bibitem[RW01]{reiner2001linear}
Victor Reiner and Volkmar Welker, \emph{Linear syzygies of stanley-reisner
  ideals}, Mathematica Scandinavica (2001), 117--132.

\bibitem[Sri89]{srinivasan1989algebra}
Hema Srinivasan, \emph{Algebra structures on some canonical resolutions},
  Journal of Algebra \textbf{122} (1989), no.~1, 150--187.

\bibitem[SZ93]{sturmfels1993maximal}
Bernd Sturmfels and Andrei Zelevinsky, \emph{Maximal minors and their leading
  terms}, Advances in mathematics \textbf{98} (1993), no.~1, 65--112.

\bibitem[Van21a]{vandebogert2021linear}
Keller VandeBogert, \emph{Linear strands supported on regular {CW} complexes},
  Submitted. arXiv preprint arXiv:2102.01114 (2021).

\bibitem[Van21b]{vandebogert2020vanishing}
\bysame, \emph{Vanishing of {A}vramov obstructions for products of sequentially
  transverse ideals}, Submitted. arXiv preprint arXiv:2011.11665 (2021).

\bibitem[Van22]{vandebogert2021products}
\bysame, \emph{Products of ideals and {G}olod rings}, To Appear, Proceedings of
  the AMS (2022).

\end{thebibliography}
\addcontentsline{toc}{section}{Bibliography}

\end{document}